\documentclass[12pt,a4paper]{amsart}
\usepackage[utf8]{inputenc}
\usepackage[T1]{fontenc}
\usepackage{amsmath,amsfonts,amssymb,amsthm,mathrsfs}
\usepackage{color}
\usepackage{marginnote}
\usepackage{constants}
\newcommand{\be}{\begin{equation}} 
\newcommand{\ee}{\end{equation}}
\newcommand{\bea}{\begin{eqnarray}} 
\newcommand{\eea}{\end{eqnarray}}
\newcommand{\bean}{\begin{eqnarray*}} 
\newcommand{\eean}{\end{eqnarray*}}

\def\mn{|\!\!|}

\def\mn2{|\!\!|_{M^{d/2}}}

\newtheorem{theorem}{Theorem}

\newtheorem{lemma}[theorem]{Lemma}

\theoremstyle{definition}

\theoremstyle{remark}
\newtheorem{remark}[theorem]{Remark}

\numberwithin{equation}{section}
\numberwithin{theorem}{section}

\newcommand{\normA}[2][]{\left\lVert #2 \right\rVert_{#1}}
\newconstantfamily{St1}{symbol=C}

\author[P. Knosalla]{Piotr Knosalla}
\address[P. Knosalla]{
Institute of Physics, Opole University, Oleska 48, 45-052 Opole, Poland\\
ORCID: 0000-0002-3594-0938}
\email{piotr.knosalla@uni.opole.pl}

\thanks{}

\title[Chemotaxis-consumption-
growth model]{
The impact of Robin boundary condition on a chemotaxis-consumption-growth model}

\newcommand\norm[1]{\left\lVert#1\right\rVert}

\begin{document}
\begin{abstract}
We investigate a parabolic-elliptic chemotaxis-consumption-growth system with a Robin boundary condition imposed on the signal. First, we analyse the steady state problem, then  we show that the solutions of the system are global and uniformly bounded in time in any space dimension. Next, under smallness assumption on the boundary data, we show that the solutions converge to non-constant steady states as time tends to infinity. \end{abstract}



\date{\today}
\maketitle

\baselineskip=16pt

\section{Introduction}

Consider a bounded domain $\Omega\subset \mathbb{R}^n$, for $n\geq 1$ with the boundary of the class $C^{2,\alpha}$, which is filled with a fluid.
A colony of bacteria occupies the region $\Omega$ and its distribution is described by the density function $u(x,t)$, for $x\in\bar{\Omega}$ and $t\geq 0$. The bacteria walk randomly and follow the highest concentration of nutrient that is dissolved in the fluid and distributed according to the density $v(x,t).$ We assume that the colony grows logistically, the nutrient diffuses in the fluid and is consumed by the bacteria. In this paper we are interested in the following parabolic-elliptic system with non-homogeneous Robin boundary condition imposed on the signal,  

\begin{equation}\label{systPE}
\left\{ \begin{array}{l}
 u_t=\nabla\cdot(\nabla u-u\nabla v)+\lambda u-\mu u^2, \\
 \partial_\nu u-u\partial_\nu v|_{\partial\Omega}=0,\\
 \Delta v=vu,\\
 \partial_\nu v|_{\partial\Omega}=(\gamma-v)g\\
 u(\cdot,0)=u_0>0,
 \end{array} \right.
\end{equation}
which describes the evolution of the model.
Here $g\in C^{1,\alpha}(\partial\Omega)$ is a positive function and $\lambda,\mu,\gamma$ are  positive parameters. Function $u_0$ is a given initial density of bacteria.
The system \eqref{systPE} can be seen as a simplification of the system proposed by Tuval et al. in \cite{Tuval}  with neglected fluid velocity and high diffusion of the signal.  

Lankeit and Wang analysed a fully parabolic version of the system \eqref{systPE} under homogeneous Neumann condition in \cite{LankeitWang}, they obtained global existence of solutions and their convergence to the unique nonnegative steady state $(\frac{\lambda}{\mu}, 0),$ when time goes to infinity. 
Introducing non-homogeneous boundary conditions in this system makes the model more realistic.
The steady state problem of \eqref{systPE} under the Dirichlet boundary condition was investigated in \cite{KW}, and recently the global existence and asymptotic behaviour of solutions in two space dimensions were considered in \cite{KnLa}.  

The aim of this paper is to show that the results similar to \cite{KW} hold for the stationary problem of \eqref{systPE}, and due to the Robin boundary condition, global existence and convergence to steady state hold in any dimension of space.

 The stationary problems of the variants of system \eqref{systPE} without growth were considered in \cite{BrLank} under the Robin boundary condition, and in \cite{LeeWangYang, AhnLankeit, KNadz, Knos2} under the Dirichlet or Neumann boundary condition. The question of global existence and/or asymptotic behaviour of solutions was studied in \cite{TWin} under the no-flux boundary condition and in \cite{FLM, AhnKangLee}, \cite{LankeitWinkler, HongWang, YangAhn}, \cite{KNos,Knos1} under the Robin, Dirichlet or non-homogenous Neumann boundary condition. The local asymptotic stability of the steady states was investigated in \cite{LiLi}.  More references on chemotaxis-consumption systems can be found in the review paper by Lankeit and Winkler, \cite{LankeitWinkler}.

\section{Main results}

We begin our analysis of system \eqref{systPE} by considering the stationary problem. Namely, we are interested in the question of the existence and uniqueness of positive time-independent solutions of \eqref{systPE}. By $U(x)$, $V(x)$ we denote the stationary density of bacteria and chemoattractant respectively which solve the following system  
\begin{equation}\label{systSPE}
\left\{ \begin{array}{l}
 0=\nabla\cdot(\nabla U-U\nabla V)+\lambda U-\mu U^2, \\
 \partial_\nu U-U\partial_\nu V|_{\partial\Omega}=0,\\
 \Delta V=VU,\\
 \partial_\nu V|_{\partial\Omega}=(\gamma-V)g. 
 \end{array} \right.
\end{equation}
The pair $(U,V)$ is called a positive solution of \eqref{systSPE} if both its components are positive.
\begin{remark}
The pair $(0,\gamma)$ is the only nonnegative constant solution of \eqref{systSPE}. This solution is dynamically unstable.
\end{remark} 
Our first result is the following.
\begin{theorem}\label{systSPE:ex:thm}
There exists a positive solution $(U,V)\in(C^{2,\alpha}(\overline{\Omega}))^2$ of (\ref{systSPE}). Any such solution satisfy
\begin{equation}\label{std:thm:U}
\frac{\lambda}{\mu}e^{V(x)-\gamma}\leq U(x)\leq\frac{\lambda}{\mu}e^{V(x)},\quad\text{ for all }x\in\overline{\Omega}
\end{equation}
\begin{equation}\label{std:thm:V}
0< V(x)<\gamma,\quad\text{ for all }x\in\overline{\Omega}
\end{equation}
Moreover there exists a constant $\gamma^\star$ such that for $\gamma\in(0,\gamma^\star)$ the pair $(U,V)$ is a unique positive solution of \eqref{systSPE}.
\end{theorem}   
Our second result concerns the global existence and uniform boundedness of solutions of \eqref{systPE}.
\begin{theorem}\label{Evol1:Thm}
For $u_0>0$, $u_0\in W^{1,p}(\Omega)$ with $p>n$,  there exists a unique globally defined solution $(u,v)$ of \eqref{systPE} satisfying
\begin{align*}
u&\in C([0,\infty);W^{1,p}(\Omega))\cap C^{2,1}(\overline{\Omega}\times(0,\infty)),\\ 
v&\in C([0,\infty);C^{2,\sigma_*}(\overline{\Omega}))\cap C^{0,1}((0,\infty);C^{2,\sigma_*}(\overline{\Omega})),
\end{align*} 
for some $\sigma_*\in(0,1)$ with
$$
\sup_{t\in[0,\infty)}\normA[L^\infty(\Omega)]{u(t)}<\infty,
$$

$$
u(x,t)>0,\quad \gamma>v(x,t)>0\quad\text{ for } (x,t)\in\bar{\Omega}\times[0,\infty).
$$
\end{theorem} 
Our third result is about the convergence of solutions to steady states.  
\begin{theorem}\label{Evol2:Thm}
Let $(u,v)$ be a solution of \eqref{systPE} with $u_0>0$, and $(U,V)$ a positive solution of \eqref{systSPE}, then there exists a positive constant ${\gamma^\star}'(\lambda,\mu,\normA[L^{\infty}(\Omega)]{u})$ such that for $\gamma\in(0,{\gamma^\star}')$ and any $s\in[2,\infty)$  we have
$$
\normA[L^s(\Omega)]{u(t)-U}\rightarrow 0,\quad \normA[W^{2,s}(\Omega)]{v(t)-V}\rightarrow 0
$$
and the rate of convergence is exponential.
\end{theorem}

\section{Stationary problem}

The proof of Theorem \ref{systSPE:ex:thm} is based on the ideas used in \cite{KW}. The change of variable $W(x):=U(x)e^{-V(x)}$ transforms the system \eqref{systSPE} into the following one
\begin{equation}\label{systSPE:M}
\left\{ \begin{array}{l}
 0=\nabla\cdot(e^{V}\nabla W)+\lambda We^V-\mu \left(We^V\right)^2, \\
 \partial_\nu W|_{\partial\Omega}=0,\\
 \Delta V=VWe^V,\\
 \partial_\nu V|_{\partial\Omega}=(\gamma-V)g. 
 \end{array} \right.
\end{equation}
We split the above system into two separate problems, namely we consider the function $V[W]\in C^{1,\alpha}(\overline{\Omega})$ to be a weak solution of 
\begin{equation}
\left\{ \begin{array}{l}\label{st:V}
 \Delta V=VWe^V,\\
 \partial_\nu V|_{\partial\Omega}=(\gamma-V)g 
 \end{array} \right.
\end{equation}   
for a given function $W\in C^0(\overline{\Omega})_{+}=\{f\in C^0(\overline{\Omega}):f(x)\geq 0\,\text{ for all } x\in\overline{\Omega}\}$, and the function $W[V]\in C^{2,\alpha}(\overline{\Omega})$ to be a solution of 
\begin{equation}\label{stW:aux}
\left\{ \begin{array}{l}
 0=\nabla\cdot(e^{V}\nabla W)+\lambda We^V-\mu \left(We^V\right)^2, \\
 \partial_\nu W|_{\partial\Omega}=0, 
 \end{array} \right.
\end{equation}
for a given function $V\in C^{1,\alpha}(\overline{\Omega})$. Recall the following result for the latter problem, which comes from \cite{KW}.
\begin{lemma}\label{W:ex:est}
Let $V\in C^{1,\alpha}(\overline{\Omega})$ be a given function. There exists a unique positive solution $W[V]\in C^{2,\alpha}(\overline{\Omega})$ of the problem (\ref{stW:aux}). This solution satisfies 
\begin{equation}\label{W:est}
\min_{x\in\overline{\Omega}}\left(e^{-V(x)}\right)\frac{\lambda}{\mu}\leq W[V](x)\leq \max_{x\in\overline{\Omega}}\left(e^{-V(x)}\right)\frac{\lambda}{\mu}
\end{equation}
for all $x\in\overline{\Omega}$. The mapping $V\mapsto W[V]$ is a continuous function from $C^{1,\alpha}(\overline{\Omega})$ to $C^{0}(\overline{\Omega})$.
\end{lemma}
In the proof of Theorem \ref{systSPE:ex:thm} we will also need the following result. 
\begin{lemma}\label{V:ex:lem}
For each $W\in C^0(\overline{\Omega})_+$ there exists a unique solution $V[W]\in C^{1,\alpha}(\overline{\Omega})$, this solution depends continuously on $W$, that is, the mapping $W\mapsto V[W]$ is a continuous function from $C^0(\Omega)_+$ to $C^{1,\alpha}(\overline{\Omega})$. This solution satisfies 
\begin{equation}\label{V:weak:est}
0\leq V[W](x)\leq \gamma , 
\end{equation}
and if $W\in C^\alpha(\overline{\Omega})$ and $W\not \equiv 0$, then we have
\begin{equation}\label{V:class:est}
0<V[W](x)<\gamma\quad\text{ for all }x\in\bar{\Omega}.
\end{equation}
\end{lemma}

\begin{proof}
\begin{enumerate}
\item {\it A priori estimates.}\\
Let $W\in C^0(\overline{\Omega})_+$ be a given function and consider the general problem
\begin{equation}\label{V:Fp:nlin}
\left\{\begin{array}{lr}
\Delta V=\sigma V e^{V} W,&x\in\Omega\\
 \partial_\nu V=(\sigma\gamma-V) g,&x\in\partial\Omega
\end{array}\right.
\end{equation}
for $\sigma\in[0,1]$. Of course, when $\sigma=1$ we obtain problem \eqref{st:V}.

 Due to elliptic regularity we expect that the solution $V[W]$ of \eqref{V:Fp:nlin} (\eqref{st:V}) can be at most in $C^{1,\beta}(\bar{\Omega})$ for any $\beta\in(0,1)$. Suppose  that $V\in C^{1,\beta}(\overline{\Omega})$ is a weak solution of \eqref{V:Fp:nlin}, then $z_1:=V-\sigma\gamma$ weakly solves the problem
$$
\Delta z_1=\sigma(z_1+\sigma\gamma)e^{z_1+\sigma\gamma}W,\quad \partial_\nu|_{\partial\Omega}z_1=-z_1g.
$$
Testing this equation by $(z_1)_+^2$ one obtains
$$
0\geq -\left(\int_{\partial\Omega}z_1g(z_1)_+^2+\int_\Omega \sigma(z_1+\sigma\gamma)e^{z_1+\sigma\gamma}W(z_1)_+^2\right) =2\int_\Omega|\nabla z_1|^2(z_1)_+ .
$$
This inequality implies that $(z_1)_+\equiv 0$ a.e. in $\Omega$, and thus $V(x)\leq \sigma\gamma$ for all $x\in\Omega$. Now, the function $z_2=-V$  solves the problem
$$
-\Delta z_2=-\sigma z_2e^{-z_2}W,\quad \partial_\nu|_{\partial\Omega}z_2=-(\sigma\gamma+z_2)g.
$$
If we test this equation by $(z_2)_+^2$, then we get
$$
2\int_\Omega|\nabla z_2|^2(z_2)_+ =-\left(\int_{\partial\Omega}(\sigma\gamma+z_2)g(z_2)_+^2 +\int_\Omega \sigma z_2e^{-z_2}W(z_2)_+^2 \right)\leq 0,
$$
Similarly to preceding case we have that $(z_2)_+\equiv 0$ a.e in $\Omega$, and so $V(x)\geq 0$. We obtain the following estimate for solution of \eqref{V:Fp:nlin}, 
\begin{equation}\label{V:sig:est}
0\leq V\leq\sigma\gamma
\end{equation}
for $\sigma\in[0,1]$, and particularly \eqref{V:weak:est}.

Now, suppose that $W\in C^{\alpha}(\overline{\Omega})$, then by the Schauder estimates, \cite[Theorem 6.31]{GiT}, $V\in C^{2,\alpha}(\overline{\Omega})$  and we can use the maximum principle. If $W\not \equiv 0$ and $\sigma\neq 0$, then the function $V$ is non constant solution of the equation 
$$
\Delta V+cV=0\quad \text{ in }\Omega
$$
with $c=-\sigma e^VW\leq 0$, thus by the strong maximum principle \cite[Theorem 3.5]{GiT} can not achieve a nonpositive minimum in $\Omega$, that is $V>0$ in $\Omega$. Using this we should have $\partial_\nu V(x_0)\leq 0$ at any point $x_0\in\partial\Omega$ where  we have $V(x_0)=0$, this is a contradiction with
$\partial_\nu V(x_0)=\sigma\gamma g(x_0)>0.$ Thus we obtain $V>0$ in $\bar{\Omega}.$ Also by the strong maximum principle we obtain that the nonnegative maximum of $V$ can not be achieved in $\Omega$. Let $x_1\in\partial\Omega$ be a point where $V$ achieves a positive maximum and suppose that $V(x_1)>\sigma\gamma$. Then, by the boundary condition, we have
$$
\partial_\nu V(x_1)=(\sigma\gamma-V(x_1))g(x_1)<0,
$$
this is a contradiction with the Hopf lemma \cite[Lemma 3.4]{GiT}. Thus, by taking $\sigma=1$, we proved \eqref{V:class:est}.  If $W\equiv 0$ and $\sigma\neq 0$, then the only solution of \eqref{V:Fp:nlin} is $V=\sigma\gamma$. When $\sigma=0$ the only solution of \eqref{V:Fp:nlin} is zero.

\item {\it Existence.}\\
We apply the Leray-Schauder theorem in the proof of existence of $V[W]$. For a given $f\in C^0(\overline{\Omega})$ we consider the operator $\mathcal{F}[f]=V$, where $V$ is a solution of the following linear problem

\begin{equation}\label{V:lin:pr}
\left\{\begin{array}{lr}
\Delta V=fe^f W,&x\in\Omega\\
 \partial_\nu V+gV=\gamma g,&x\in\partial\Omega
\end{array}\right.
\end{equation}

Thanks to \cite[Theorem 2.4.2.6]{Grisvard} there exists a unique solution $V\in W^{2,p}(\Omega),$ for any $p\in(1,\infty),$ of \eqref{V:lin:pr} which satisfies the estimate
\begin{equation}\label{Lp:est}
\begin{aligned}
\normA[W^{2,p}(\Omega)]{V}&\leq C_1\left(\normA[L^p(\Omega)]{fe^f W}+\normA[W^{1-1/p, p}(\partial\Omega)]{\gamma g}\right)\\
&\leq C_2\left(\normA[L^\infty(\Omega)]{fe^f W}+\gamma\normA[C^{1,\alpha}(\partial\Omega)]{ g}\right). 
\end{aligned}
\end{equation} 
 This estimates and Sobolev embedding gives us that $\mathcal{F}[f]\in C^{1,\beta}(\bar{\Omega})$ with any $\beta\in(0,1)$ thus $\mathcal{F}$ is a compact operator in $C^0(\bar{\Omega})$, by the Arzel{\'a}-Ascoli theorem.
Now, take a sequence of continuous functions $\{f_n\}_n$ which converges uniformly to $f$. Using \eqref{Lp:est} we obtain that
$$
\normA[W^{2,p}(\Omega)]{\mathcal{F}[f_n]-\mathcal{F}[f]}\leq C_2\normA[L^\infty(\Omega)]{\left(f_ne^{f_n}- fe^{f}\right)W}
$$ 
and continuity of $\mathcal{F}$ follows. Consider the family of fixed point equations
\begin{equation}\label{V:fp:e}
V=\sigma \mathcal{F}[V],\quad\text{ for }\sigma\in[0,1].
\end{equation}

By direct calculations we see that solving \eqref{V:fp:e} for a given $\sigma$ is equivalent with solving \eqref{V:Fp:nlin}. And by \eqref{V:sig:est} we obtain that for any $\sigma\in[0,1]$ solution of \eqref{V:fp:e} satisfy
$$
\normA[L^\infty(\Omega)]{V}\leq \gamma
$$ 
and the Leray-Schauder theorem gives us existence of solution $V[W]$ of \eqref{st:V}.

\item {\it Uniqueness and continuous dependence.}\\ 
\end{enumerate}
Let $V_1,V_2$ be two distinct solutions of (\ref{st:V}) Their difference $\tilde{V}=V_1-V_2$ satisfy the equation
$$
\Delta \tilde{V}=W(V_1e^{V_1}-V_2e^{V_2}),\quad \partial_\nu \tilde{V}|_{\partial\Omega}=-\tilde{V}g.
$$
We multiply this equation by $\tilde{V}$ and integrate over $\Omega$, then 
\begin{equation}\label{uniq:equality}
-\left(\int_{\partial\Omega}\tilde{V}^2g+\int_\Omega|\nabla \tilde{V}|^2\right)=\int_\Omega W (V_1e^{V_1}-V_2e^{V_2})\tilde{V}\,\geq 0
\end{equation}
where the last inequality follows the fact that the function $xe^x$ is increasing, and $W\geq 0$. It follows from \eqref{uniq:equality} and $\tilde{V}\in C^{1,\beta}(\overline{\Omega})$ that $\tilde{V}=0,$ a contradiction. Now we show that $W\mapsto V[W]$ is a continuous mapping from $C^0(\overline{\Omega})_+$ into itself. Let $\{W_n\}_n\subset C^0(\overline{\Omega})_+$ be a sequence converging uniformly to $W$ and assume that $V[W_n]\not \rightarrow V[W]$, that is there exists some $\varepsilon>0$ and a sub sequence $\{W_{n_k}\}_{k}$ for which we have $\norm{V[W_{n_k}]-V[W]}_{C^{1,\alpha}(\overline{\Omega})}\geq \varepsilon$ for all $k\in\mathbb{N}.$ By \eqref{V:weak:est} and \eqref{Lp:est} we know that $$\norm{V[W_{n_k}]}_{C^{1,\beta}(\overline{\Omega})}\leq C_3\gamma\left( e^\gamma\sup_k  \normA[L^\infty(\Omega)]{W_{n_k}}+\normA[C^{1,\alpha}]{g}\right)$$ for $\beta>\alpha$, and so by Arzel{\'a}-Ascoli theorem there exists another sub sequence $\{V[W_{n_{k_i}}]\}_i$ which converges in $C^{1,\alpha}(\overline{\Omega})$ to some $V_\infty$, which is a weak solution of \eqref{st:V}, and by  uniqueness $V_\infty=V[W],$ a contradiction.  
  
\end{proof}

Now we are ready to prove Theorem \ref{systSPE:ex:thm}.
\begin{proof}[Proof of Theorem \ref{systSPE:ex:thm}]
Consider the space $X=\{f\in C^0(\bar{\Omega}):e^{-\gamma}\frac{\lambda}{\mu}\leq f\leq\frac{\lambda}{\mu}\}$ and define the operator
$$
\mathcal{T}[f]:=W[V[f]]
$$
for $f\in X$. It is clear that $X$ is closed and convex. The operator $\mathcal{T}$ is continuous, as a composition of two continuous operators, and compact, since $\mathcal{T}[f]\in C^{2,\alpha}(\bar{\Omega})$. It also follows from Lemma \ref{W:ex:est} that $\mathcal{T}[X]\subset X$. By the Schauder fixed point theorem, \cite{Evans}, there exists a fixed point $W\in X$ that, in fact, is in $C^{2,\alpha}(\bar{\Omega})$. Estimates \eqref{std:thm:U}, \eqref{std:thm:V} follow from \eqref{W:est} and \eqref{V:class:est}. The proof of uniqueness of the positive solution for small $\gamma$ is postponed to Lemma \ref{uniq:std}.
\end{proof}

Before we give a proof of the uniqueness, we need the following lemma.
\begin{lemma}
Let $\Omega\subset\mathbb{R}^n$ be a bounded domain with Lipschitz boundary, then for any $f\in H^1(\Omega)$ and any $\varepsilon>0$ we have
 \begin{equation}\label{trace:est}
\norm{f}_{L^2(\partial\Omega)}^2\leq \varepsilon\norm{\nabla f}_{L^2(\Omega)}^2+c_1(\varepsilon)\norm{f}_{L^2(\Omega)}^2
\end{equation}
and
\begin{equation}\label{GN:iter:ineq}
\normA[L^2(\Omega)]{f}^2\leq \varepsilon \normA[L^2(\Omega)]{\nabla f}^2+c_1'(\varepsilon)\normA[L^1(\Omega)]{f}^2 
\end{equation}
\end{lemma}
\begin{proof}
From the trace inequality  
$$
\norm{f}_{L^2(\partial\Omega)}\leq C_T\norm{f}_{L^2(\Omega)}^\frac{1}{2}\norm{f}_{H^1(\Omega)}^\frac{1}{2}
$$ 
(see \cite[Theorem III.2.19]{BoyerFabrie} for the proof) we obtain \eqref{trace:est} with $c_1(\varepsilon)=\frac{C_T^4}{4\varepsilon}+\varepsilon.$ Now we prove the second inequality. From the Gagliardo-Nirenberg \cite[Theorem 10.1]{Friedman} inequality we have 
$$
\normA[L^2(\Omega)]{f}^2\leq C_{GN} \normA[L^2(\Omega)]{\nabla f}^{2\theta}\normA[L^1(\Omega)]{f}^{2(1-\theta)}+C_{GN}\normA[L^1(\Omega)]{f}^2
$$
with $\theta=\frac{n}{2+n}$. By the Young inequality we obtain \eqref{GN:iter:ineq}
with $c_1'(\varepsilon)= C_{GN}\max\{1,{C_{GN}}^{\frac{n}{2}}\}(1+\varepsilon^{-\frac{n}{2}})=C_4(1+\varepsilon^{-\frac{n}{2}}).$

\end{proof}

\begin{lemma}\label{uniq:std}
There exists a positive constant $\gamma^\star(\normA[L^\infty]{g},\lambda,\mu)$ such that the positive solution $(U,V)$ is unique for $\gamma <\gamma^\star$.
\end{lemma}
\begin{proof}
Assume that we have two distinct positive solutions, $(U_i,V_i)$ for $i=1,2$ and let $\tilde{U}:=U_1-U_2$, $\tilde{V}:=V_1-V_2$. The difference $\tilde{V}$ satisfy,
$$
\Delta\tilde{V}-U_2\tilde{V}=\tilde{U}V_1,\quad \partial_\nu\tilde{V}|_{\partial\Omega}=-\tilde{V}g,
$$
and by the Young's inequality, \eqref{V:weak:est} we obtain
\begin{align*}
\int_{\partial\Omega}\tilde{V}^2g+\int_\Omega |\nabla \tilde{V}|^2+\int_\Omega U_2 \tilde{V}^2 &=-\int_\Omega\tilde{U}\tilde{V}V_1\\
&\leq \varepsilon_1 \int_\Omega\tilde{V}^2+ \gamma^2 \frac{1}{4\varepsilon_1}\int_\Omega\tilde{U}^2.
\end{align*}
Recalling that $\inf_{\Omega}U_i>\frac{\lambda}{\mu}e^{-\gamma}$ for $i=1,2$ and choosing $\varepsilon_1=\frac{\lambda}{2\mu}e^{-\gamma}$ we obtain, 
\begin{equation}\label{V:tilda:est}
\int_{\partial\Omega}\tilde{V}^2g+\int_\Omega |\nabla \tilde{V}|^2+\frac{\lambda}{2\mu}e^{-\gamma}\int_\Omega \tilde{V}^2 \leq \gamma^2 \frac{\mu}{2\lambda}e^{\gamma}\int_\Omega \tilde{U}^2
\end{equation}

 The difference $\tilde{U}$ satisfy
$$
0=\nabla\cdot\left(\nabla\tilde{U}-(\tilde{U}\nabla V_1+U_2\nabla\tilde{V})\right)+\lambda\tilde{U}-\mu\tilde{U}(U_1+U_2)
$$
$$
(\partial_\nu \tilde{U}-(\tilde{U}\partial_\nu V_1+U_2\partial_\nu\tilde{V}))|_{\partial\Omega}=0
$$
Multiplying this equation by $2\tilde{U}$, integrating over $\Omega$, since $V_i,U_i>0$ for $i=1,2$ one gets
\begin{align*}
2\int_\Omega|\nabla \tilde{U}|^2 &+2\mu\int_\Omega\tilde{U}^2\left((U_1+U_2)-\frac{\lambda}{\mu}\right)=\int_\Omega\nabla \tilde{U}^2\cdot\nabla V_1+2\int_\Omega U_2\nabla\tilde{U}\cdot\nabla \tilde{V}\\
&=\int_{\partial\Omega}\tilde{U}^2\partial_\nu V_1-\int_\Omega\tilde{U}^2U_1V_1+2\int_\Omega U_2\nabla\tilde{U}\cdot\nabla \tilde{V}\\
&\leq \gamma\normA[L^\infty]{g}\int_{\partial\Omega}\tilde{U}^2+\int_\Omega |\nabla\tilde{U}|^2+\left(\frac{\lambda}{\mu}e^{\gamma}\right)^2\int_\Omega |\nabla \tilde{V}|^2.  
\end{align*}   
Using \eqref{V:tilda:est}, \eqref{trace:est} with $\varepsilon=\frac{1}{2}$
 we obtain,
 
\begin{equation}\label{uniq:ineq}
\mathcal{F}_1(\gamma)\int_\Omega|\nabla \tilde{U}|^2 + 2\mu\mathcal{F}_2(\gamma)\int_\Omega\tilde{U}^2 \leq 0
\end{equation} 
with
$$
\mathcal{F}_1(\gamma)=1-\gamma\frac{\normA[L^\infty]{g}}{2}
$$
$$
\mathcal{F}_2(\gamma)=\frac{\lambda}{\mu}\left(2e^{-\gamma}-1\right)-\gamma^2\frac{e^\gamma}{4\lambda}\left(\frac{\lambda}{\mu}e^{\gamma}\right)^2-\gamma\frac{\normA[L^\infty]{g}c_1\left(\frac{1}{2}\right)}{2\mu}
$$
Since $\lim_{\gamma\searrow 0}\mathcal{F}_2(\gamma)=\frac{\lambda}{\mu}>0$, there exists $0<\gamma^\star(\lambda,\mu, \normA[L^\infty]{g})<\frac{2}{\normA[L^\infty]{g}}$ such that $\mathcal{F}_1(\gamma)>0$ and $\mathcal{F}_2(\gamma)>0$ for $\gamma\in(0,\gamma^\star).$ Thus, when $\gamma <\gamma^\star$ inequality  \eqref{uniq:ineq} implies $\tilde{U}=0$, and  \eqref{V:tilda:est} gives $\tilde{V}=0,$ a contradiction.
\end{proof}

\section{Global existence}
In this section we show the global existence and uniform boundedness of solutions of (\ref{systPE}). To do that, we use the Moser-Alikakos iteration procedure \cite{Alikakos1, Alikakos2}. We start with the following result concerning the local existence of solutions.  

\begin{lemma}[Local existence]\label{Loc:Ex}
Let $u_0>0$, $u_0\in W^{1,p}(\Omega)$ with $p>n$,  then there exists a unique solution $(u,v)$ of \eqref{systPE} satisfying
\begin{align*}
u&\in C([0,T_{\max});W^{1,p}(\Omega))\cap C^{2,1}(\overline{\Omega}\times(0,T_{\max})),\\ 
v&\in C([0,T_{\max});C^{2,\sigma_*}(\overline{\Omega}))\cap C^{0,1}((0,T_{\max});C^{2,\sigma_*}(\overline{\Omega})),
\end{align*} 
for some $\sigma_*\in(0,1)$, where $T_{\max}\in(0,\infty]$ is the maximal time of existence of solution. If $T_{\max}<\infty$, then 
\begin{equation}\label{blow:cond}
\limsup_{t\nearrow T_{\max}}\normA[L^\infty]{u(t)}=+\infty.
\end{equation}
\end{lemma}

\begin{proof}
Proceeding like in the proof of \cite[Theorem 5.1]{KnLa} with minor modifications we can show the existence of a unique solution $(u,v)$ of the following modified problem
$$
\left\{ \begin{array}{l}
 u_t=\nabla\cdot(\nabla u-u\nabla v)+\lambda u-\mu u^2, \\
 \partial_\nu u-u\partial_\nu v|_{\partial\Omega}=0,\\
 \Delta v=v(u)_+,\\
 \partial_\nu v|_{\partial\Omega}=(\gamma-v)g\\
 u(0)=u_0>0 
 \end{array} \right.
$$
with the above stated regularity and blow-up condition \eqref{blow:cond}. Now, Lemma \ref{comp:lemma} implies that $u>0$ in $\bar{\Omega}\times [0,T_{\max})$, thus $(u)_+=u,$ and the couple $(u,v)$ is a unique solution of \eqref{systPE}.
\end{proof}

According to the preceding lemma to prove the global existence of solutions, we need to exclude the scenario described by condition \eqref{blow:cond}, we do this in the following lemma. 
\begin{lemma}[Global boundedness of solutions]\label{Glo:ex}
Let $(u,v)$ be a classical solution of \eqref{systPE}  defined on $[0,T_{\max})$, then there exists a time-independent constant
$c_5^*$ such that, 
\begin{equation}
\sup_{t\in[0,T_{\max})}\normA[L^\infty(\Omega)]{u(t)}<c_5^*.
\end{equation}
\end{lemma}

\begin{proof}[Proof of Lemma \ref{Glo:ex}]  
Integrating the $u$-equation and using no-flux boundary condition we obtain,
\begin{align*}
\frac{d}{dt}\normA[L^1(\Omega)]{u(t)}&=\lambda \normA[L^1(\Omega)]{u(t)}-\mu\int_\Omega u(\cdot,t)^2\\
&\leq \lambda \normA[L^1(\Omega)]{u(t)}-\frac{\mu}{|\Omega|}\normA[L^1(\Omega)]{u(t)}^2,
\end{align*}
and by ODE comparison we see that 
\begin{equation}\label{L1:uni}
\normA[L^1(\Omega)]{u(t)}\leq \max\left\{\normA[L^1(\Omega)]{u_0}, \frac{|\Omega|\lambda}{\mu}\right\}\quad \text{ for all } t\in[0,T_{\max}). 
\end{equation}
 Now we assume that $p\geq 2$,
 multiply $u$-equation by $pu^{p-1}$ and integrate in space, 
\begin{align*}
\frac{d}{dt}\norm{u(t)}_{L^p(\Omega)}^p &+\frac{4(p-1)}{p}\norm{\nabla u^{\frac{p}{2}}(t)}_{L^2(\Omega)}^2\\ 
&=(p-1)\int_\Omega \nabla (u^p) \cdot\nabla v 
 +\lambda p \norm{u(t)}_{L^p(\Omega)}^p-\mu p \norm{u(t)}_{L^{p+1}}^{p+1}\\
&=(p-1)\left(\int_{\partial\Omega}u^p\partial_\nu v -\int_\Omega u^p\Delta v\right)+\lambda p \norm{u(t)}_{L^p(\Omega)}^p-\mu p \norm{u(t)}_{L^{p+1}(\Omega)}^{p+1}\\
&\leq (p-1)\int_{\partial\Omega}u^p\gamma g\, +\lambda p \norm{u(t)}_{L^p(\Omega)}^p-\mu p \norm{u(t)}_{L^{p+1}(\Omega)}^{p+1}.
\end{align*} 
To estimate the boundary integral term we use \eqref{trace:est} for $u^{\frac{p}{2}}$ with $\varepsilon=\frac{2}{p\gamma \normA[L^\infty]{g}}$ and we get 
\begin{align*}
(p-1)\int_{\partial\Omega}u^p\gamma g&\leq \frac{2(p-1)}{p}\normA[L^2(\Omega)]{\nabla u^{\frac{p}{2}}}^2+(p-1)c_1\left(\frac{2}{p\gamma \normA[L^\infty]{g}}\right)\gamma\normA[L^\infty]{g} \normA[L^p(\Omega)]{u}^p.
\end{align*}
We have
\begin{align*}
\frac{d}{dt}\norm{u(t)}_{L^p(\Omega)}^p &+\frac{2(p-1)}{p}\norm{\nabla u^{\frac{p}{2}}(t)}_{L^2(\Omega)}^2+(p-1)p\normA[L^p(\Omega)]{u(t)}^p\\ 
&\leq \left((p-1)c_1\left(\frac{2}{p\gamma \normA[L^\infty]{g}}\right)\gamma\normA[L^\infty]{g} +\lambda p+p(p-1)\right)\normA[L^p(\Omega)]{u(t)}^p\\
&\leq p(p-1)c^*_1\normA[L^p(\Omega)]{u(t)}^p,
\end{align*}
where $c_1^*=\frac{C_T^4(\gamma\normA[L^\infty]{g})^2}{8}+\lambda+2$ is independent of $p$. Now, using \eqref{GN:iter:ineq} we obtain
\begin{align*}
p(p-1)c^*_1\normA[L^p(\Omega)]{u}^p&\leq \frac{2(p-1)}{p}\normA[L^p(\Omega)]{\nabla u^\frac{p}{2}}^2+p(p-1)c_1'\left(\frac{2}{p^2c_1^*}\right)c_1^*\normA[L^1(\Omega)]{u^\frac{p}{2}}^2\\
&\leq\frac{2(p-1)}{p}\normA[L^p(\Omega)]{\nabla u^\frac{p}{2}}^2+c_2^*p(p-1)(1+p^n)\normA[L^1(\Omega)]{u^\frac{p}{2}}^2
\end{align*}
with $c_2^*=C_4\cdot\max\left\{1,\left(\frac{c_1^*}{2}\right)^{\frac{n}{2}}\right\}\cdot c_1^*$ independent of $p$. Gathering all the above inequalities we end up with the following one
\begin{equation}
\frac{d}{dt}\normA[L^p(\Omega)]{u(t)}^p+p(p-1)\normA[L^p(\Omega)]{u(t)}^p\leq c_2^*p(p-1)(1+p)^n \normA[L^1(\Omega)]{u^\frac{p}{2}(t)}^2
\end{equation}
which can be rewritten into
$$
\frac{d}{dt}\left(e^{p(p-1)t}\normA[L^p(\Omega)]{u(t)}^p\right)\leq e^{p(p-1)t}p(p-1)c_2^*(1+p)^n \normA[L^1(\Omega)]{u^\frac{p}{2}(t)}^2,
$$ 
and after integration on $[0,t]$ we get,
\begin{align*}
\normA[L^p(\Omega)]{u(t)}^p&=e^{-p(p-1)t}\left(\normA[L^p(\Omega)]{u_0}^p+c_2^*(1+p)^n\int_0^te^{p(p-1)s}p(p-1) \normA[L^1(\Omega)]{u^\frac{p}{2}(s)}^2\right)\\
&\leq e^{-p(p-1)t}\left(\normA[L^p(\Omega)]{u_0}^p+c_2^*(1+p)^n\sup_{s\in[0,t]}\normA[L^1(\Omega)]{u^\frac{p}{2}(s)}^2\int_0^te^{p(p-1)s}p(p-1) \right)\\
&\leq \normA[L^p(\Omega)]{u_0}^p+c_2^*(1+p)^n\sup_{s\in[0,t]}\normA[L^1(\Omega)]{u^\frac{p}{2}(s)}^2.
\end{align*}
Denoting 
$$
K(p):=\max\left\{\normA[L^\infty(\Omega)]{u_0},\sup_{s\in[0,t]}\normA[L^p(\Omega)]{u(s)}\right\}
$$
we obtain the following recurrence inequality
\begin{equation}\label{K:rec:ineq}
K(p)\leq [c_3^*(1+p)^n]^{1/p} K(p/2)
\end{equation} 
with $c^*_3=|\Omega|+c_2^*.$ Now we claim that the sequence $\{K(2^j)\}_{j\in\mathbb{N}}$ is bounded. Indeed, from \eqref{K:rec:ineq} we have
\begin{align*}
K(2^j)&\leq [c_3^*(1+2^j)^n]^{\frac{1}{2^j}}[c_3^*(1+2^{j-1})^n]^{\frac{1}{2^{j-1}}}\cdots[c_3^*(1+2)^n]^{\frac{1}{2}}K(1)\\
&\leq {(c_3^*)}^{\sum_{i=1}^j\frac{1}{2^i}}[2^{n(j+1)}]^{\frac{1}{2^j}}[2^{nj}]^{\frac{1}{2^{j-1}}}\cdots[2^{2n}]^{\frac{1}{2}}K(1)\\
&={(c_3^*)}^{\sum_{i=1}^j\frac{1}{2^i}}2^{n\sum_{i=1}^j\frac{i+1}{2^i}}K(1)\\
&\leq c_3^* 2^{4n}K(1)=c_4^*K(1).
\end{align*}
Finally, taking the limit with respect to $j$ and using \eqref{L1:uni}, we conclude that
\begin{align*}
\normA[L^\infty(\Omega)]{u(t)}&\leq c_4^* \max\left\{\normA[L^\infty(\Omega)]{u_0},\sup_{s\in[0,t]}\normA[L^1(\Omega)]{u(s)}\right\}\\
&\leq c_4^*\max\left\{\normA[L^\infty(\Omega)]{u_0},\normA[L^1(\Omega)]{u_0},\frac{\lambda|\Omega|}{\mu}\right\} =:c_5^* 
\end{align*}
for all $t\in[0,T_{\max}).$

\end{proof}
\begin{remark}
By quick inspection of the above proof we infer that $c_5^*$ is an increasing function of $\gamma$, and thus $c_5^*(\gamma)$ is bounded as $\gamma\searrow 0.$
\end{remark}

\begin{proof}[Proof of Theorem \ref{Evol1:Thm}]
The proof follows from Lemma \ref{Loc:Ex}, Lemma \ref{Glo:ex}. The positivity of $u(x,t)$ is shown in Lemma \ref{comp:lemma}, using this and the maximum principle (as in the proof of Lemma \ref{V:ex:lem})  we obtain estimates for $v(x,t).$
\end{proof}

\section{Asymptotic behavior of solutions}

Before we prove the convergence of solution to the steady state, we recall the following result being a consequence of the maximum principle.
\begin{lemma}\label{comp:lemma}
Let $(u,v)$ be a solution of \eqref{systPE} with $u_0>0$ defined on the maximal time interval $[0,T_{\max})$. Then the solution $y(t)$ of the following initial vale problem 
$$
\left\{\begin{array}{l}y'=\lambda y-(\mu+\gamma)y^2\\
y(0)=\inf_{x\in\overline{\Omega}} u_0
\end{array}\right.
$$
is a sub-solution of $u$, that is 
\begin{equation}\label{u:low:bound}
u(x,t)\geq y(t),\quad\text{ for all } (x,t)\in\overline{\Omega}\times[0,T_{\max}).
\end{equation}
\end{lemma}
\begin{proof}
We prove the statement by contradiction. Let $y_\varepsilon(t)$ be a solution of the problem
$$
\left\{\begin{array}{l}y_\varepsilon'=\lambda y_\varepsilon-(\mu+\gamma)y_\varepsilon^2\\
y_\varepsilon(0)=\inf_{x\in\overline{\Omega}} u_0-\varepsilon
\end{array}\right.
$$
with $\varepsilon\in(0,\inf_{x\in\Omega}u_0),$ and $w(x,t):=u(x,t)-y_\varepsilon(t).$ Function $w(x,t)$ solves the following parabolic equation
\begin{equation}\label{maxprinc:aux}
w_t=\Delta w-\nabla w\cdot\nabla v+\lambda w-(\mu+v)u^2+(\mu+\gamma)y_\varepsilon^2\quad\text{ in }\Omega\times(0,T_{\max})
\end{equation}
Suppose that the set $S=\{(x,t)\in\overline{\Omega}\times[0,T_{\max}):w(x,t)< 0\}$ is not empty. It is clear that $w(x,0)>0$ for all $x\in\overline{\Omega}$, and thus, by continuity of $w,$ there exists $(x_0,t_0)\in\partial S$ ($w(x_0,t_0)=0$) such that $w(x,t)>0$ for $(x,t)\in\overline{\Omega}\times[0,t_0)$ and $w(x,t_0)\geq 0$ for $x\in\overline{\Omega}$. The function $w(x,t_0)$ is nonnegative and thus reaches a minimum at $x_0\in\overline{\Omega}$. If we assume that $x_0\in\partial\Omega$, then since $g>0$ on $\partial\Omega$ we have
$$
0\geq\partial_\nu w(x_0,t_0)=\partial_\nu u(x_0,t_0)=u(x_0,t_0)\partial_\nu v(x_0,t_0)=y_\varepsilon(t_0)(\gamma-v(x_0,t_0))g(x_0)>0,
$$  
a contradiction. 
Assuming $x_0\in\Omega$ we have $w_t(x_0,t_0)\leq 0$, $\Delta w(x_0,t_0)\geq 0,$ $\nabla w(x_0,t_0)=0$ and from \eqref{maxprinc:aux} we obtain
$$
0\geq w_t(x_0,t_0)\geq (\gamma-v(x_0,t_0))y_\varepsilon^2(t_0)>0,
$$
since $0<v(x,t)<\gamma$, again a contradiction. Since $\varepsilon$ is arbitrary, we obtain \eqref{u:low:bound}.  
\end{proof}

Now we proceed to the proof of Theorem \ref{Evol2:Thm}.

\begin{proof}[Proof of Theorem \ref{Evol2:Thm}]
We consider the differences $\tilde{u}=u-U$, $\tilde{v}=v-V$. Since $\tilde{v}$ is a solution of
\begin{equation}\label{v:diff:evol:eq}
\Delta\tilde{v}-\tilde{v}U=v\tilde{u},\quad \partial_{\nu}\tilde{v}|_{\partial\Omega}=-\tilde{v}g,
\end{equation}
proceeding like in the proof of Lemma \ref{uniq:std} we obtain the inequality,
\begin{equation}\label{v:diff:evol:ineq}
\int_{\partial\Omega}\tilde{v}^2g+\int_\Omega|\nabla \tilde{v}|^2+\frac{\lambda}{2\mu}e^{-\gamma}\int_\Omega \tilde{v}^2\leq \gamma^2\frac{\mu}{2\lambda}e^\gamma \int_\Omega \tilde{u}^2
\end{equation}
The difference $\tilde{u}$ fulfills the equation
\begin{equation}
\tilde{u}_t=\nabla\cdot(\nabla \tilde{u}-(\tilde{u}\nabla v+U\nabla \tilde{v}))+\mu \tilde{u}\left(\frac{\lambda}{\mu}-
u-U\right) 
\end{equation}
with the boundary condition
$$
\left(\partial_\nu \tilde{u}-(\tilde{u}\partial_{\nu} v+ U\partial_{\nu}\tilde{v})\right)|_{\partial\Omega}=0.
$$
Multiplying this equation by $2\tilde{u}$, integrating on $\Omega$ and using $u,v>0$ we obtain
\begin{align*}
\frac{d}{dt}\int_\Omega\tilde{u}^2 &+2\int_\Omega|\nabla \tilde{u}|^2+2\mu\int_\Omega\tilde{u}^2\left(U+u-\frac{\lambda}{\mu}\right)=\int_\Omega\nabla\tilde{u}^2\cdot\nabla v+2\int_\Omega U\nabla \tilde{u}\cdot\nabla \tilde{v}\\
&=\int_{\partial\Omega}\tilde{u}^2\partial_\nu v-\int_\Omega \tilde{u}^2\Delta v+ 2\int_\Omega U\nabla \tilde{u}\cdot\nabla \tilde{v}\\
&\leq \gamma \normA[L^\infty]{g}\int_{\partial\Omega}\tilde{u}^2+\int_\Omega|\nabla\tilde{u}|^2+ \left(\frac{\lambda}{\mu}e^\gamma\right)^2\int_\Omega |\nabla\tilde{v}|^2  
\end{align*}

Using \eqref{v:diff:evol:ineq}, \eqref{trace:est}  we arrive at
\begin{equation}\label{ev:inequality}
\frac{d}{dt}\int_\Omega\tilde{u}^2 +\mathcal{F}_{1}(\gamma)\int_\Omega|\nabla \tilde{u}|^2+2\mu \mathcal{F}_{e2}(\gamma)\int_\Omega\tilde{u}^2\leq 0
\end{equation}
with $\mathcal{F}_{1}(\gamma)$ defined in the proof of Lemma \ref{uniq:std}, and
\begin{align*}
\mathcal{F}_{e2}(\gamma)&=\frac{\lambda}{\mu}(e^{-\gamma}-1)+y-\gamma^2\frac{e^\gamma}{4\lambda}\left(\frac{\lambda}{\mu}e^\gamma\right)^2-\gamma\frac{\normA[L^\infty]{g}c_1\left(\frac{1}{2}\right)}{2\mu}\\
&\geq \frac{\lambda}{\mu}(e^{-\gamma}-1)+\min\left\{\inf_{x\in\overline{\Omega}}u_0,\frac{\lambda}{\mu+\gamma}\right\}-\gamma^2\frac{e^\gamma}{4\lambda}\left(\frac{\lambda}{\mu}e^\gamma\right)^2-\gamma\frac{\normA[L^\infty]{g}c_1\left(\frac{1}{2}\right)}{2\mu}\\
&=\mathcal{F}_{e2*}(\gamma)
\end{align*}
Since $\lim_{\gamma\searrow 0}\mathcal{F}_{e2*}(\gamma)=\min\left\{\inf_{x\in\overline{\Omega}}u_0,\frac{\lambda}{\mu}\right\}>0$ 
we see that there exists $\frac{2}{\normA[L^\infty]{g}}>{\gamma^\star}'(\lambda,\mu,\normA[L^{\infty}(\Omega)]{u})>0$ such that for $\gamma< {\gamma^\star}'$ we have
$$
\mathcal{F}_{e2*}(\gamma)>0,\quad \mathcal{F}_{1}(\gamma)>0,
$$
Therefore from \eqref{ev:inequality} we infer
$$
\frac{d}{dt}\int_\Omega\tilde{u}^2 +2\mu \mathcal{F}_{e2*}(\gamma)\int_\Omega\tilde{u}^2\leq 0,
$$
thus
$$
\normA[L^2(\Omega)]{\tilde{u}(t)}\leq e^{-\mu\mathcal{F}_{e2*}(\gamma) t }\normA[L^2(\Omega)]{\tilde{u}_0},
$$ 
and by the H{\" o}lder inequality, Lemma \ref{Glo:ex} for $s\geq 2$ we have
\begin{align*}
\normA[L^s(\Omega)]{\tilde{u}}&\leq \normA[L^2(\Omega)]{\tilde{u}}^{\frac{2}{s}}\normA[L^\infty(\Omega)]{\tilde{u}}^{1-\frac{2}{s}} \\
&\leq e^{-\frac{2}{s}\mu\mathcal{F}_{e2*}(\gamma) t }\normA[L^2(\Omega)]{\tilde{u}_0}^{\frac{2}{s}}\left(c_5^*+\frac{\lambda}{\mu}e^{\gamma}\right)^{1-\frac{2}{s}} \\
&= e^{-\frac{2}{s}\mu\mathcal{F}_{e2*}(\gamma) t }c^*_6. \\
\end{align*}
Using this and elliptic regularity for equation \eqref{v:diff:evol:eq} (\cite[Theorem 2.4.2.6]{Grisvard}) we find a constant $C_4(s)$  such that,
$$
\normA[W^{2,s}(\Omega)]{\tilde{v}(t)}\leq \gamma C_4(s)c^*_6e^{-\frac{2}{s}\mu\mathcal{F}_{e2*}(\gamma) t }=c_7^*e^{-\frac{2}{s}\mu\mathcal{F}_{e2*}(\gamma) t }. 
$$
\end{proof}

\footnotesize
\def\cprime{$'$}

\end{document}